\documentclass[12pt]{amsart}
\usepackage{amssymb,amsmath,amsthm,fullpage,enumerate}
\usepackage{pdfsync}
\usepackage{color}

\newtheorem{thm}{Theorem}[section]
\newtheorem{prop}[thm]{Proposition}

\newtheorem{cor}[thm]{Corollary}

\newtheorem*{theorem*}{Theorem}
\theoremstyle{definition}
\newtheorem{defn}[thm]{Definition}

\theoremstyle{remark}
\newtheorem{rem}[thm]{Remark}

\newcommand{\Z}{\mathbb Z}

\newcommand{\C}{\mathbb C}
\DeclareMathOperator{\supp}{supp}

\DeclareMathOperator{\Rre}{Re}

\DeclareMathOperator{\Ran}{Range}

\newcommand{\bd}{\textrm{b}}

\newcommand{\p}{\partial}

\newcommand{\z}{\bar z}

\newcommand{\dbar}{\bar\partial}
\newcommand{\dbars}{\bar\partial^*}
\newcommand{\dbarst}{\bar\partial^*_t}
\newcommand{\dbarb}{\bar\partial_b}
\newcommand{\dbarbs}{\bar\partial^*_b}
\newcommand{\dbarbst}{\bar\partial^*_{b,t}}
\newcommand{\Boxb}{\Box_b}
\newcommand{\vp}{\varphi}

\newcommand{\nn}{\nonumber}
\newcommand{\ep}{\epsilon}
\newcommand{\I}{\mathcal{I}}
\newcommand{\norm}{|\hspace{-1pt}\|}

\DeclareMathOperator{\Tr}{Tr}

\newcommand{\la}{\langle}
\newcommand{\ra}{\rangle}
\renewcommand{\H}{\mathcal H}

\newcommand{\abs}[1]{\left\vert#1\right\vert}

\begin{document}

\author{Phillip S.\ Harrington, Marco M.\ Peloso and Andrew S.\ Raich}

\thanks{The first author is partially supported by NSF grant DMS-1002332 and
the third author is partially supported by NSF grant DMS-0855822}
\thanks{This paper was written while the second author was visiting
  the University of Arkansas. He wishes to thank this institution for
  its hospitality and and for providing a very pleasant working environment.}

\address{Department of Mathematical Sciences, SCEN 301, 1 University
  of Arkansas, Fayetteville, AR 72701}
\email{psharrin@uark.edu \\  araich@uark.edu}

\address{(permanent address) Departimento di Matematica, Via
  C. Saldini 50, 20133 Milano, Italy}
\email{marco.peloso@unimi.it}

\subjclass[2010]{32W10, 32W05, 35N15, 32V20, 32Q28}

\keywords{$\dbarb$, close range, Kohn's weighted theory, CR-manifold,
  hypersurface type, tangential Cauchy-Riemann operator, $Y(q)$, weak
  $Y(q)$, $Z(q)$, weak $Z(q)$,
complex Green operator, $\dbar$-Neumann operator, Stein manifolds}

\title{Regularity equivalence of the Szeg\"o projection and the complex Green operator}
\maketitle

\begin{abstract}
In this paper we prove that on a CR manifold of hypersurface
type
that satisfies the  weak $Y(q)$ condition,
the complex Green operator $G_q$ is exactly (globally) regular if and
only if the Szeg\"o projections $S_{q-1}, S_q$ and a third orthogonal projection $S'_{q+1}$ are
exactly (globally) regular.  The projection  $S'_{q+1}$ is closely
related to the Szeg\"o projection $S_{q+1}$ and actually coincides
with it if the space of harmonic $(0,q+1)$-forms is trivial.

This result extends the important and by now classical result by
H. Boas and E. Straube on the equivalence of the regularity of the
$\dbar$-Neumann operator and the Bergman projections on a smoothly bounded
pseudoconvex domain.

We also prove an extension of this result to the case of bounded
smooth domains satisfying the weak $Z(q)$
condition on a Stein manifold.
\end{abstract}

\section*{Introduction}
The goal of this article is to discuss the general principle that the
combination of an appropriate weighted theory,  a Hodge decomposition,
and the $L^2$ regularity of $\dbarb$ (resp., $\dbar$) provides the
tools
to prove the equivalence of regularity in the Sobolev scale between
the complex Green operator (resp., the $\dbar$-Neumann operator) and
the Szeg\"o projection (resp., the Bergman projection).

H. Boas and E. Straube first observed the equivalence
of the regularity of the Bergman projection and the $\dbar$-Neumann
operator on smooth, bounded pseudoconvex domains in $\C^n$.
In \cite{BoSt90} they proved the following theorem.
\begin{theorem*} {\rm {\bf (Boas and Straube)}}
Let $\Omega$ be a  smooth, bounded pseudoconvex domain in $\C^n$.  Let
$1\le q\le n$.  Then the $\dbar$-Neumann operator $N_q$ on $(0,q)$-forms is
exactly regular if and only if the three Bergman projections
$P_{q-1},\, P_q$, and $P_{q+1}$ are exactly regular.

The corresponding statement holds with the words ``exactly regular''
replaced by the words ``globally regular''.
\end{theorem*}
Recall that an operator is exactly regular if it preserves all $L^2$ Sobolev spaces and is globally regular it preserves $C^\infty$ functions (or forms).

In this paper in particular we address the question of whether such a
theorem has a counterpart in the case of the Szeg\"o projection and
the complex Green operator, (see Sections \ref{sec:notation} and
\ref{sec:results}
for precise definitions).  One of the main results of this paper
contains the following theorem as a special case.
\begin{theorem*}
Let $\Omega$ be a  smooth, bounded pseudoconvex domain in $\C^n$ and
let $M$ denote its boundary.  Let $G_q$ denote the complex Green
operator and $S_q$ the Szeg\"o projection on $(0,q)$-forms on $M$,
$1\le q\le n-2$.  Then the operator $G_q$  is
exactly regular if and only if the three Szeg\"o  projections
$S_{q-1},\, S_q$, and $S_{q+1}$ are exactly regular.

The corresponding statement holds with the words ``exactly regular''
replaced by the words ``globally regular''.
\end{theorem*}

Specifically,  in this paper
we  study the cases of the complex Green operator $G_q$ on
embedded CR manifolds of hypersurface type that satisfy the weak $Y(q)$
condition and the $\dbar$-Neumann operator
on domains in a Stein manifold that satisfy the weak $Z(q)$
condition.

The  required  estimates and weighted theory are proven by
the first and third authors in \cite{HaRa11} and \cite{HaRa12},
respectively,
and the results in this article can be thought of as a consequence of
the techniques of \cite{BoSt90} and the estimates in
\cite{HaRa11,HaRa12}.

We write the paper from the point of view of CR manifolds of
hypersurface type and only indicate the changes that are needed
to obtain the results for the $\dbar$-Neumann operator on weakly
$Z(q)$ domains in Stein manifolds. \medskip

Let $M^{2n-1}\subseteq\C^N$ be a $C^\infty$ compact, orientable CR
manifold $N\geq n$. We say that $M$ is of \emph{hypersurface type} if
the CR-dimension of $M$ is $n-1$ so that the complex tangent bundle
of $M$ splits into a complex subbundle and one totally real
direction. The $\dbarb$-complex on $M$ is obtained by restricting the de
Rham complex on $M$ to the conjugate of the complexification of the complex subbundle.

When $M$ is the boundary of a pseudoconvex domain, closed range of
$\dbarb$ on $L^2_{p,q}(M)$ for $0\leq p \leq n$ and $0 \leq q \leq
n-1$
was proved by Shaw and Boas-Shaw \cite{Sha85,BoSh86}. Independently,
Kohn also proved closed range for $\dbarb$ at all form levels and
established the weighted theory
in \cite{Koh86}. Nicoara generalized Kohn's results to the case of CR
manifolds of hypersurface type \cite{Nic06}. Harrington and Raich
further generalized \cite{Nic06} by investigating closed range
and the weighted theory for $\dbarb$ on $(0,q)$-forms for a fixed $q$
(in this case, $p$ is irrelevant and they take $p=0$ for simplicity).
They called their condition weak $Y(q)$ and developed the most general
version of it in \cite{HaRa12}. Condition $Y(q)$ is well known
to be the natural generalization of strict pseudoconvexity for
dealing with $(0,q)$-forms on $M$ for a fixed $q$.
See also \cite{AhBaZa06,Zam08} for conditions related to, but stronger
than, weak $Y(q)$. \medskip

The paper concludes with a discussion of how to adapt the argument for
the $\dbar$-Neumann operator and Bergman projection on a smooth,
bounded domain in a Stein manifold. The argument follows the general
argument for
the complex Green operator and Szeg\"o projection with some minor (and
well-known) modifications. Harrington and Raich \cite{HaRa12} develop
the $L^2$ and weighted Sobolev theory (for $-\frac 12 \leq s\leq 1$)
under the hypotheses that $\Omega\subset M$ is $C^3$, bounded, and
satisfies weak $Z(q)$. In this paper, we discuss the generalization of
the weighted theory for $s\geq 1$ when $\Omega$ is smooth
and bounded. The $L^2$ and weighted $L^2$ theories for $\dbar$ on
pseudoconvex domains in Stein manifolds are now classical and were
established by H\"ormander \cite{Hor65} and Kohn \cite{Koh73}.
\medskip

The outline of the paper is as follows. We set our notation in
Section \ref{sec:notation}, state the main results in Section
\ref{sec:results},  and prove our results in Section
\ref{sec:proofs}. We conclude
with a discussion of Stein manifolds in Section \ref{sec:Stein}.
\medskip

%
%
\section{Notation}\label{sec:notation}

Throughout the paper, we denote by $M$ a smooth, compact, embedded
and orientable CR
manifold of dimension $2n-1$ and hypersurface type.  We refer to
\cite{Bog91} for the theory of CR manifolds.

\subsection{The Levi form and weak $Y(q)$} Let $T^{p,q}(M)$ denote the collection
of $(p,q)$-vectors and $\Lambda^{p,q}(M)$ the set of $(p,q)$-forms on
$M$.
The induced CR-structure has a local orthonormal basis $L_1,\dots,
L_{n-1}$ for the $(1,0)$-vector fields in a neighborhood $U$ of each
point $x\in M$. Let $\omega_1,\dots,\omega_{n-1}$ be the dual basis of
$(1,0)$-forms that satisfy $\langle \omega_j, L_k\rangle
=\delta_{jk}$. Then $\bar L_1,\dots, \bar L_{n-1}$ is a local
orthonormal basis for the $(0,1)$-vector fields with dual basis
$\bar\omega_1,\dots,\bar\omega_{n-1}$
in $U$. Also, the tangent bundle $T(U)$ is spanned by
$L_1,\dots,L_{n-1}$, $\bar L_1,\dots,\bar L_{n-1}$, and an additional
vector field $T$ taken to be
purely imaginary (so $\bar T = -T$).

Since $M$ is orientable, there exists a
global,  purely
imaginary $1$-form $\gamma$   on $M$ that
annihilates $T^{1,0}(M)\oplus T^{0,1}(M)$ and is normalized so that
$\langle \gamma, T \rangle =-1$.

\begin{defn}\label{defn:Levi form}
The \emph{Levi form at a point $x\in M$} is the Hermitian form given
by $-\langle \gamma_x, [L,\bar L']\rangle$ 
where $L, L'\in T^{1,0}_x(U)$, and $U$ is a neighborhood of $x\in M$.
\end{defn}

We remark that $-\langle \gamma_x, [L,\bar L']\rangle = \langle d\gamma, L\wedge\bar L'\rangle$
since $\gamma$ annihilates $T^{1,0}(M)\oplus T^{0,1}(M)$.


Recall that $M$ is \emph{pseudoconvex} if, for some orientation of $\gamma$, the Levi form is positive semi-definite at all
$x\in M$ and \emph{strictly pseudoconvex} if, for some orientation of $\gamma$, the Levi form is positive definite at all
$x\in M$.

When $q$ is fixed, strict pseudoconvexity is not necessary
to prove $1/2$ estimates for the
$\dbar$-Neumann operator. Instead, the optimal condition is $Z(q)$
(see, e.g., \cite{FoKo72,ChSh01}).
$M$ is said to satisfy $Z(q)$, $1\leq q \leq n-1$, if the Levi form
associated with $M$ has at least $n-q$ positive eigenvalues
or at least $q+1$ negative eigenvalues. $M$ is said to satisfy
$Y(q)$, $1\leq q \leq n-1$,
if $M$ satisfies $Z(q)$ and $Z(n-1-q)$. The necessity of the symmetric
requirements for $\dbarb$ at levels $q$ and $n-1-q$
stems from the  duality between $(0,q)$-forms and $(0,n-1-q)$-forms
(see \cite{FoKo72} or \cite{RaSt08} for details).

Our definition of weak $Z(q)$ follows \cite{HaRa12}.
\begin{defn}\label{defn:weak Z(q)}
Let $M\subset \C^n$ be a smooth, compact, orientable CR manifold of
hypersurface type. We say that $M$ satisfies
\emph{weak $Z(q)$}
if there exists a real bivector $\Upsilon\in
T^{1,1}(M)$ that satisfies:
\begin{enumerate}\renewcommand{\labelenumi}{(\roman{enumi})}
 \item $\abs{\omega}^2\geq (i\omega\wedge\bar\omega)(\Upsilon)\geq 0$,
   for all $\omega\in \Lambda^{1,0}(M)$;

 \item $\mu_1+\cdots+\mu_q-d\gamma(\Upsilon)\geq 0$, where
   $\mu_1,\ldots,\mu_{n-1}$  are the eigenvalues of the Levi form
   in increasing order;

 \item $ \inf_{M} \{ |q-\Tr(\Upsilon)|\} >0$.

\end{enumerate}
As above, $M$ satisfies \emph{weak $Y(q)$ at x} if $M$ satisfies weak
$Z(q)$ at $x$ and weak $Z(n-1-q)$ at $x$.
\end{defn}
\begin{rem}
\label{rem:upsilon_matrix}
  In local coordinates, $\Upsilon$ may be identified with an $(n-1)\times(n-1)$ Hermitian matrix $(a_{jk})$ via $\Upsilon=\sum_{j,k}ia_{jk}\bar{L}_k\wedge L_j$.
\end{rem}

%
\subsection{Weak $Z(q)$ and the basic estimate}\label{subsec:weak z(q) and basic estimate}
In this part, we provide motivation and commentary on the weak $Z(q)$ condition.

Let $\Omega\subset\C^n$ be a smooth, bounded domain. Let
$\I_q = \{J = (j_1,\dots,j_q) : 1 \leq j_1 < \cdots < j_q \leq n\}$.
For $f, g \in L^2_{0,q}(\Omega)$, define
\[
\big( f, g \big)_t = \sum_{J\in\I_q} \int_\Omega f_J(z)\overline{g_J(z)} e^{-t|z|^2}\, dV(z)
\]
and $\|f\|_{t,L^2(\Omega)}^2 = \big(f,f\big)_t$. Let $\dbarst$ be the
$L^2$ adjoint of $\dbar$ with respect to the $(\cdot,\cdot)_t$
sesquilinear product. Let $\bd\Omega$ be the boundary of $\Omega$,
$\rho$ a defining function for
$\Omega$ with $|\nabla \rho|=1$ on $\bd\Omega$, and $d\sigma$ be the
induced surface area measure on $\bd\Omega$.
A classical version of the basic identity (or Kohn-Morrey formula) is
\begin{multline}\label{eqn:classic basic iden}
\|\dbar f \|_{t,L^2(\Omega)}^2 + \|\dbarst f\|_{t,L^2(\Omega)}^2  =
\sum_{J\in\I_q}\sum_{j=1}^n\int_\Omega \Big| \frac{\p
  f_J}{\p\z_j}\Big|^2 e^{-t|z|^2}\, dV + qt\|f\|_{t,L^2(\Omega)}^2  \\
+\sum_{K\in\I_{q-1}}\sum_{j,k=1}^n \int_{\bd\Omega}
\frac{\p^2\rho(z)}{\p z_j\p\z_k} u_{jK} \overline{u_{kK}}
e^{-t|z|^2}\, d\sigma(z).
\end{multline}
See \cite[Proposition 2.4]{Str10} for a proof. A closed range estimate
for $\dbar$ follows from this identity if the boundary integral is
positive and $t>0$. If $\Omega$ is pseudoconvex (or at least the sum of any $q$ eigenvalues of the Levi form is nonnegative), then the boundary integral will be positive.

When
$\Omega$ is not pseudoconvex, then (\ref{eqn:classic basic iden}) is
not necessarily a useful equality. For example, if $\Omega$ is an
annular region between two pseudoconvex domains, i.e.,
$\Omega = \Omega_1 \setminus \overline{\Omega_2}$ where
$\Omega_1\supset\Omega_2$ and both domains are pseudoconvex, then near
$\bd\Omega_2$, it is helpful to integrate the
$(\frac{\p f_J}{\p\z_j},\frac{\p f_J}{\p\z_j})_t$ terms by parts. If
we set $L^{t}_j = \frac{\p}{\p z_j} - t\z_j = e^{t|z|^2}\frac{\p}{\p z_j} e^{-t|z|^2}$ and
$\rho_{j\bar k} = \frac{\p^2\rho}{\p  z\p\z_k}$,
\begin{multline}\label{eqn:classic basic iden -- pseudoconcave}
\| \dbar f\|_{t,L^2(\Omega)}^2 + \| \dbars_t f \|_{t,L^2(\Omega)}^2
=  \sum_{J\in \I_q}\sum_{j,k=1}^{n}\big\|L^t_j
f_J\big\|_{t,L^2(\Omega)}^2 - t(n-q) \|f\|_{t,L^2(\Omega)}^2 \\
+\sum_{I\in\I_{q-1}} \sum_{j,k=1}^n \int_{\bd\Omega}\rho_{j\bar k}
f_{jI} \overline{ f_{kI}} e^{-t|z|^2}d\sigma
- \sum_{J\in \I_q}   \int_{\bd\Omega} \Tr(\rho_{j\bar k})  |f_J|^2
e^{-t|z|^2} d\sigma  + O (\| f \|_{t,L^2(\Omega)}^2),
\end{multline}
Equation \eqref{eqn:classic basic iden -- pseudoconcave} works where
$\bd\Omega$ is pseudoconcave since the eigenvalues of the Levi form
are nonpositive. We also need $t<0$ for a closed range estimate.

The $(q-1)$-pseudoconcave property stems from the idea that we do not
have to integrate by parts all of the $(0,1)$ vector fields. For example,
if we arranged the eigenvalues of the Levi form in increasing order
and had a coordinate
system where the $j$th coordinate was associated with the $j$th
eigenvalue of the Levi from (e.g., if the Levi form was
diagonalizable), then an effective identity would be a combination of
(\ref{eqn:classic basic iden}) and (\ref{eqn:classic basic iden --
  pseudoconcave}). Certain $(1,0)$ and $(0,1)$ vector fields appear and
we do not subtract the full trace of the Levi form. In fact, we get a
basic identity of the form
\begin{multline}\label{eqn:classic basic iden -- (q-1)-pseudoconvex}
\| \dbar f\|_{t,L^2(\Omega)}^2 + \| \dbars_t f \|_{t,L^2(\Omega)}^2
\\=  \sum_{J\in \I_q}\sum_{k=m+1}^{n} \Big\|\frac{\partial
  f_J}{\partial\bar z_k}\Big\|_{t,L^2(\Omega)}^2 + \sum_{J\in
  \I_q}\sum_{j=1}^{m} \big\|L_j f_J\big\|_{t,L^2(\Omega)}^2 + t (q-m)
\| f\|_{t,L^2(\Omega)}^2\\
+\int_{\bd\Omega} \Bigg[\sum_{I\in\I_{q-1}} \sum_{j,k=1}^n \rho_{j\bar k}
f_{jI} \overline{f_{kI}} e^{-t|z|^2}d\sigma
- \sum_{j=1}^m  \rho_{j\bar j}  |f|^2 e^{-t|z|^2} \Bigg]
d\sigma  + O (\| f \|_{t,L^2(\Omega)}^2).
\end{multline}
The sign of $t$ depends on whether $m>q$ or $m<q$, and this depends on
how many eigenvalues of the Levi form are negative. The only value
that $m$ is not allowed to take is $m=q$. Zampieri's
$(q-1)$-pseudoconvexity
is a condition that requires a vector bundle of dimension $m$ so that the boundary integral in \eqref{eqn:classic
  basic iden -- (q-1)-pseudoconvex} is always a positive term
and $m<q$. In \cite{HaRa11}, Harrington and Raich permitted the
case $m>q$, which is useful when dealing with annular regions.

In \cite{HaRa12}, Harrington and Raich introduced a matrix $\Upsilon$ (relative to a choice of basis for $T^{p,q}(M)$; see Remark \ref{rem:upsilon_matrix}) that governs the
integration by
parts. In the pseudoconvex convex case, $\Upsilon$ is the $0$ matrix (no
integration by parts needed).
In the pseudoconcave case, $\Upsilon = I$, the identity matrix, since
every $(0,1)$ vector field needs to be integrated by parts. In the
$(q-1)$-pseudoconcave case (or weak $Z(q)$ case with the definition
from
\cite{HaRa11}), $\Upsilon$ is diagonal and has the form
\[
\Upsilon = \begin{pmatrix} I_m & 0 \\ 0 & 0 \end{pmatrix},
\]
where $I_m$ is the $m\times m$ identity matrix. In looking at the
basic identities, \eqref{eqn:classic basic iden}, \eqref{eqn:classic
  basic iden -- pseudoconcave}, and
\eqref{eqn:classic basic iden -- (q-1)-pseudoconvex}, Harrington and
Raich observed three items in trying to form the matrix $\Upsilon$:
\begin{enumerate}[	i.]
\item We need $0 \leq \Upsilon \leq I$ or the sum of the $(0,1)$ and
  $(1,0)$ vector fields may not be positive.

\item $\Upsilon$ must be chosen so that the boundary integral is positive.

\item $\Upsilon$ cannot cause the $L^2$ norm of $f$ that is multiplied
  by $t$ to vanish. This is the $t(q-m)\|f\|_{t,L^2(\Omega)}^2$ term in
  \eqref{eqn:classic basic iden -- (q-1)-pseudoconvex}.
\end{enumerate}
Given the requirements on $\Upsilon = (\Upsilon^{j\bar k})$, they
formulated the weak $Z(q)$ condition for domains in a Stein manifold. In the case of an
embedded CR manifold of hypersurface type, this definition becomes
Definition \ref{defn:weak Z(q)}. The basic identity for a smooth,
bounded pseudoconvex domains $\Omega\subset\C^n$ is then
\begin{align*}
&\| \dbar f\|_{t,L^2(\Omega)}^2 + \| \dbars_t f \|_{t,L^2(\Omega)}^2
=  \sum_{J\in \I_q}\sum_{j,k=1}^{n}\left((\delta_{jk}-\Upsilon^{\bar
    kj})\frac{\partial f_J}{\partial\bar z_k},\frac{\partial
    f_J}{\partial\bar z_j}\right)_t
+\sum_{J\in \I_q}\sum_{j,k=1}^{n}\left(\Upsilon^{\bar kj}L^t_j
  f_J,L^t_k  f_J\right)_t \\
&+\sum_{I\in\I_{q-1}} \sum_{j,k=1}^n \int_{\bd\Omega}\rho_{j\bar k}
f_{jI} \overline{ f_{kI}} e^{-t|z|^2}d\sigma
- \sum_{j,k=1}^n   \int_{b\Omega} \Upsilon^{\bar kj}\rho_{j\bar k}
|f|^2 e^{-t|z|^2} d\sigma \\
&+ 2\Rre\Bigg\{  \sum_{J\in
  \I_q}\sum_{j,k,\ell=1}^{n}\left(\frac{\partial\Upsilon^{\bar
      kj}}{\p\z_k}\Upsilon^{\bar j\ell}
  L^t_\ell  f_J,f_J\right)_t
  -\sum_{J\in \I_q}\sum_{j,k,\ell=1}^{n} 				
   \left(\frac{\partial\Upsilon^{\bar kj}}{\partial z_j}(\delta_{k
       \ell}-\Upsilon^{\bar\ell k})\frac{\partial f_J}{\partial\bar
       z_\ell},f_J\right)_t \Bigg\}\\
&+ \sum_{J\in\I_q} t \big((q-\Tr(\Upsilon)) f_J, f_J\big)_t
+ O (\| f \|_{t,L^2(\Omega)}^2),
\end{align*}
where $O(\|f\|_{t,L^2(\Omega)}^2)
\leq C(\|\Upsilon\|_{C^2(\bar\Omega)}+ \| \Upsilon\|_{C^2(\bar\Omega)}^2)
\| f \|_{t,L^2(\Omega)}$.
This identity includes \eqref{eqn:classic basic iden},
\eqref{eqn:classic basic iden -- pseudoconcave}, and
\eqref{eqn:classic basic iden -- (q-1)-pseudoconvex} as special cases,
as discussed above.

\subsection{Norms}
We follow the notation from \cite[Section 3]{HaRa11}. We set
\[
(\vp,\phi)_t = \int_M \phi\bar\vp e^{-t|z|^2}\, d\sigma.
\]
In particular, $t=0$ is the standard, unweighted $L^2$ inner product
and has norm $\|\vp\|_{L^2(M)}^2 = (\vp,\vp)_0$.

We follow the setup for the microlocal analysis in \cite{Rai10,HaRa11}.
Since $M$ is compact, there exists a finite cover $\{ U_\nu\}_{\nu}$
so each $U_{\nu}$ has a special boundary system and
can be parameterized by a hypersurface in $\C^n$ ($U_\nu$ may be
shrunk as necessary). To set up the microlocal analysis,
we need to define the appropriate pseudodifferential operators on each
$U_{\nu}$. Let
$\xi = (\xi_1,\dots,\xi_{2n-2},\xi_{2n-1}) = (\xi',\xi_{2n-1})$ be the
coordinates in Fourier space so that
$\xi'$ is dual to the part of $T(M)$ in the maximal complex subspace
(i.e., $T^{1,0}(M)\oplus T^{0,1}(M)$) and
$\xi_{2n-1}$ is dual to the totally real part of $T(M)$, i.e.,the
``bad" direction $T$. Define
\begin{align*}
\mathcal{C}^+
&= \{ \xi : \xi_{2n-1} \geq \frac 12 |\xi'| \text{ and } |\xi|\geq1\};\\
\mathcal{C}^-
&= \{\xi : -\xi\in \mathcal{C}^+\};\\
\mathcal{C}^0
&= \{\xi : -\frac 34|\xi'| \leq \xi_{2n-1}\leq \frac 34 |\xi'|\} \cup \{\xi : |\xi|\leq 1\}.
\end{align*}
Note that $\mathcal{C}^+$ and $\mathcal{C}^-$ are disjoint, but both
intersect $\mathcal{C}^0$ nontrivially. Next, we define smooth
functions on
$\{|\xi| : |\xi|^2 =1\}$. Let
\begin{align*}
\psi^+(\xi)
&= 1 \text{ when } \xi_{2n-1}\geq \frac 34|\xi'| \text{ and } \supp
\psi^+ \subset \{\xi : \xi_{2n-1}\geq \frac 12|\xi'|\}; \\
\psi^-(\xi)
&= \psi^+(-\xi); \\
\psi^0(\xi)
&\text{ satisfies } \psi^0(\xi)^2 = 1- \psi^+(\xi)^2 -
\psi^-(\xi)^2.
\end{align*}
Extend $\psi^+$, $\psi^-$, and $\psi^0$ homogeneously outside of the
unit ball, i.e., if $|\xi|\geq 1$, then
\[
\psi^+(\xi) = \psi^+(\xi/|\xi|),\ \psi^-(\xi) =
\psi^-(\xi/|\xi|),\text{ and } \psi^0(\xi) = \psi^0(\xi/|\xi|).
\]
Also, extend $\psi^+$, $\psi^-$, and $\psi^0$ smoothly inside the unit
ball so that $(\psi^+)^2+(\psi^-)^2 + (\psi^0)^2 =1$. Finally,
there exists a large constant $A>0$ that depends on $M$ (which allows
the weighted Sobolev theory to hold and whose existence is proven in
\cite{Rai10,HaRa11}) when we define,
for any $t>0$,
\[
\psi^+_t(\xi) = \psi^+(\xi/(t A)),\ \psi^-_t(\xi) = \psi^-(\xi/(t
A)),\text{ and }\psi^0_t(\xi) = \psi^0(\xi/(t A)).
\]
Next, let $\Psi^+_{t}$, $\Psi^-_{t}$, and $\Psi^0_t$ be the pseudodifferential operators of order zero with symbols
$\psi^+_t$, $\psi^-_t$, and $\psi^0_t$, respectively. The equality
$(\psi^+_t)^2+(\psi^-_t)^2 + (\psi^0_t)^2 =1$ implies that
\[
(\Psi^+_{t})^*\Psi^+_{t} + (\Psi^0_{t})^*\Psi^0_{t} + (\Psi^-_{t})^*\Psi^-_{t} = Id.
\]
We will also have use for pseudodifferential operators  that
``dominate" a given pseudodifferential operator. Let
$\psi$ be cut-off function and $\tilde\psi$ be another cut-off
function so that $\tilde\psi|_{\supp \psi} \equiv 1$. If $\Psi$ and
$\tilde\Psi$ are pseudodifferential operators with symbols $\psi$ and
$\tilde\psi$, respectively, then we say that
$\tilde\Psi$ dominates $\Psi$.

For each $U_\nu$,  we can define
$\Psi^+_{t}$, $\Psi^-_{t}$, and $\Psi^0_{t}$ to act on functions or
forms supported in $U_\nu$, so let $\Psi^+_{t,\nu}$, $\Psi^-_{t,\nu}$,
and $\Psi^0_{t,\nu}$
be the pseudodifferential operators of order zero defined on $U_\nu$,
and let $\mathcal{C}^+_\nu$, $\mathcal{C}^-_\nu$, and
$\mathcal{C}^0_\nu$ be the regions of
$\xi$-space dual to $U_\nu$ on which the symbol of each of those
pseudodifferential operators is supported. Then it follows that:
\[
(\Psi^+_{t,\nu})^*\Psi^+_{t,\nu} + (\Psi^0_{t,\nu})^*\Psi^0_{t,\nu} +
(\Psi^-_{t,\nu})^*\Psi^-_{t,\nu} = Id.
\]

Let $\{\zeta_\nu\}$ be a partition of unity
subordinate to the covering $\{U_\nu\}$ satisfying $\sum_{\nu}
\zeta_\nu^2=1$. Also, for each $\nu$, let $\tilde\zeta_\nu$ be a
cutoff function
that dominates $\zeta_\nu$ so that $\supp\tilde\zeta_\nu \subset
U_\nu$. We define
\begin{multline*}
\la \phi,\vp\ra_t
= \sum_{\nu} \Big[ \big(\tilde\zeta_\nu \Psi^+_{\nu,t}
\zeta_\nu \phi^\nu, \tilde\zeta_\nu \Psi^+_{\nu,t} \zeta_\nu \vp^\nu\big)_{\lambda_+t}
\\+ \big(\tilde\zeta_\nu \Psi^0_{\nu,t} \zeta_\nu \phi^\nu,
\tilde\zeta_\nu \Psi^0_{\nu,t} \zeta_\nu \vp^\nu\big)_0
+ \big(\tilde\zeta_\nu \Psi^-_{\nu,t} \zeta_\nu \phi^\nu,
\tilde\zeta_\nu \Psi^-_{\nu,t} \zeta_\nu
\vp^\nu\big)_{\lambda_-t}\Big],
\end{multline*}
where
\[
\lambda_+ = \begin{cases} 1 &\text{if }  \Tr\Upsilon <q    \\  -1 &
  \text{if } \Tr\Upsilon >q, \end{cases}
\]
and
\[
\lambda_- = \begin{cases} -1 &\text{if }  \Tr\Upsilon < n-1-q    \\  1
  & \text{if } \Tr\Upsilon > n-1-q .\end{cases}
\]
Set
\[
\norm\vp\norm_t^2 = \la\vp,\vp\ra_t.
\]

Let $\Lambda^s$ be the pseudodifferential operator with symbol
$(1+|\xi|^2)^{s/2}$. We set the Sobolev $s$-norm on $W^s(M)$ to be
\[
\| \vp \|_{W^s(M)}^2 = \sum_\nu \|\tilde\zeta_\nu \Lambda^s \zeta_\nu \vp^\nu \|_{L^2(M)}^2.
\]

It is shown in \cite{Nic06,Rai10} that there exist constants $c_t,C_t>0$ so that
\[
c_t \|\vp\|_{L^2(M)}^2 \leq \norm \vp \norm_t^2 \leq C_t \|\vp\|_{L^2(M)}^2
\]
and an invertible pseudodifferential operator of order 0, $F_t$, so that
\begin{equation}
\label{eq:F_t}
\la\vp,\phi\ra_t = (\vp,F_t\phi)_0.
\end{equation}

\subsection{$L^2$ theory for $\dbarb$}\label{sec:L^2 theory}
In \cite{HaRa11}, Harrington and Raich established Kohn's weighted theory for $\dbarb$. In particular, let $\dbarbst$ be the $L^2$-adjoint of $\dbarb$ in $\la\cdot,\cdot\ra_t$,
$\Box_{b,t} = \dbarb\dbarbst+\dbarbst\dbarb$, $H_{q,t}$ the projection of $L^2_{0,q}(M,e^{-t|z|^2})$ onto $\ker\dbarb\cap\ker\dbarbst$,
and $G_{q,t}$ be the relative inverse to $\Box_{b,t}$, that is, the
inverse on the orthogonal complement
of $\ker\Box_{b,t}$.  When $t=0$, we suppress the subscript.  We also know that $\dbarbs - \dbarbst$ is an operator
of order $0$ from \cite[Lemma 3.7]{Rai10}.

We have the Hodge decomposition
\[
I = \dbarb\dbarbs G_q + \dbarbs\dbarb G_q + H_q
\]
and a similar Hodge decomposition for the weighted operators. Let
$S_q:L^2_{0,q}(M)\to\ker\dbarb$ be the Szeg\"o projection.  Since
$S_q$ is self-adjoint, it follows
that $\ker S_q = (\Ran S_q)^\perp$. It is also easily checked that $S_q\dbarbs=0$, so
\[
S_q =  \dbarb\dbarbs G_q + H_q
\]
and therefore Kohn's formula
\[
S_q = I - \dbarbs\dbarb G_q
\]
holds. Since we do not know that $G_{q-1}$ exists as a continuous
operator on $L^2_{0,q-1}(M)$ (and hence cannot commute  $G_{q-1}$ with
$\dbarb$), we define
\[
S_{q-1} = I - \dbarbs G_q \dbarb.
\]
Then $S_{q-1}$ is a self-adjoint projection and hence is still an
orthogonal projection. We will continue to call $S_{q-1}$ a Szeg\"o
projection because if we had a Hodge theory for
$L^2_{0,q-1}(M)$, then $S_{q-1}$ would agree with the Szeg\"o
projection as defined above. We also set
\[
S_{q+1}' =  \dbarb G_q \dbarbs.
\]
The orthogonal projection $S_{q+1}'$ is not generically the Szeg\"o projection
because it annihilates harmonic forms.

Every formula in this section has a weighted analog.
\medskip

%
%
\section{Statements of the Main Results}\label{sec:results}
In what follows, we reserve $t\geq 0$ for the weight $\lambda_t(z) = e^{-t|z|^2}$ and
$s\geq0$ for Sobolev norms of order $s$ (defined below).

\subsection{CR manifolds of hypersurface type}\label{subsec:CR mfld results}
\begin{thm}\label{thm:regularity equivalence}
Let $M$ be a smooth, compact, embedded, CR manifold of hypersurface
type that satisfies weak $Y(q)$ for some $1 \leq q \leq n-2$. Let
$s\geq 0$. If $G_q$ is a continuous operator
on $W^{s+2}_{0,q}(M)$, then there exists a constant $C_r$ so that, for
every $u\in C^\infty(M)$,
\[
\|S_{q-1} u\|_{W^r(M)} + \|S_qu\|_{W^r(M)} + \|S_{q+1}'u\|_{W^r(M)}
\leq C_r \|G_qu\|_{W^r(M)}
\]
for $0 \leq r \leq s$.

If $S_{q-1}$, $S_q$, and $S_{q+1}'$ are continuous operators on
$W^s_{0,q-1}(M)$, $W^s_{0,q}(M)$, and $W^s_{0,q+1}(M)$, respectively,
then $G_q$ is a continuous operator
on $W^{s}_{0,q}(M)$ and there exists a constant $C_{s}$ so that for
every $u\in C^\infty(M)$,
\[
\|G_qu \|_{W^s(M)} \leq C_s \big(\|S_{q-1}u\|_{W^s(M)} +
\|S_qu\|_{W^s(M)} + \|S_{q+1}'u\|_{W^s(M)} \big).
\]
\end{thm}

\begin{cor}\label{cor:exact regularity}
Let $M$ be a smooth, compact, embedded, CR manifold of hypersurface type that satisfies weak $Y(q)$ for some $1 \leq q \leq n-2$. Then $G_q$ is exactly regular if and only if
$S_{q-1}$, $S_q$, and $S_{q+1}'$ are exactly regular.
\end{cor}

\begin{prop}\label{prop:G_q dominates}
Let $M\subset \C^N$ be a smooth, compact, embedded, CR manifold of hypersurface type that satisfies weak $Y(q)$. Let $k\in\Z$ be a positive integer. If $u$ and $G_q u$ are both in $W^{k+2}_{0,q}(M)$
and $u \perp \H_q$, then
there exists a constant $C>0$ so that
\begin{multline*}
\|\dbarbs\dbarb G_q u\|_{W^k(M)} + \|\dbarb\dbarbs G_q u\|_{W^k(M)} +
\|\dbarb G_q u\|_{W^k(M)} + \|\dbarbs G_q u\|_{W^k(M)} \\
\leq C\big( \|G_q u\|_{W^k(M)} + \|u\|_{W^k(M)} \big).
\end{multline*}
\end{prop}

Proposition \ref{prop:G_q dominates} should be compared with \cite[Lemma
3.2]{Str10}.

\subsection{Smooth, bounded domains in a Stein manifold}\label{subsec:Stein mfld results}
We have a similar group of results for smooth, bounded domains in a Stein manifold.
\begin{thm}\label{thm:N_q regularity equivalence}
Let $M$ be a Stein manifold and $\Omega\subset M$ a smooth, bounded
domain that satisfies weak $Z(q)$ for some $1 \leq q \leq n-1$. Let
$s\geq 0$. If $N_q$ is a continuous operator
on $W^{s+2}_{0,q}(\Omega)$, then there exists a constant $C_r$ so that
\[
\|P_{q-1}u\|_{W^r(\Omega)} + \|P_qu\|_{W^r(\Omega)} + \|P_{q+1}' u\|_{W^r(\Omega)}
\leq C_r \|N_q u\|_{W^r(\Omega)}
\]
for all $u\in C^\infty(\Omega)$ and
for $0 \leq r \leq s$.

If $P_{q-1}$, $P_q$, and $P_{q+1}'$ are continuous operators on
$W^s_{0,q-1}(M)$, $W^s_{0,q}(M)$, and $W^s_{0,q+1}(M)$, respectively,
then $N_q$ is a continuous operator
on $W^{s}_{0,q}(M)$ and there exists a constant $C_{s}$ so that
\[
\|N_q\|_{W^s(\Omega)}  \leq C_s \big(\|P_{q-1}\|_{W^s(\Omega)}
+ \|P_q\|_{W^s(\Omega)} + \|P_{q+1}'\|_{W^s(\Omega)} \big).
\]
\end{thm}

\begin{cor}\label{cor:N_q exact regularity}
Let $M$ be a Stein manifold and $\Omega\subset M$ a smooth, bounded
domain that satisfies weak $Z(q)$ for some $1 \leq q \leq n-1$. Then
$N_q$ is exactly regular if and only if
$P_{q-1}$, $P_q$, and $P_{q+1}'$ are exactly regular.
\end{cor}

\begin{prop}\label{prop:N_q dominates}
Let $M$ be a Stein manifold and $\Omega\subset M$ a smooth, bounded
domain that satisfies weak $Z(q)$ for some $1 \leq q \leq n-1$.
Let $k\in\Z$ be a positive integer. If $u$ and $N_q u$ are both in $W^{k+2}_{0,q}(M)$
and $u \perp \H_q$, then
there exists a constant $C>0$ so that
\begin{multline*}
\|\dbars\dbar N_q u\|_{W^k(\Omega)} + \|\dbar\dbars N_q
u\|_{W^k(\Omega)} + \|\dbar N_q u\|_{W^k(\Omega)} + \|\dbars N_q
u\|_{W^k(\Omega)}\\
 \leq C\big( \|N_q u\|_{W^k(\Omega)} +
\|u\|_{W^k(\Omega)} \big).
\end{multline*}
\end{prop}

In summary,
we have generalized the approach of \cite{BoSt90} in
several ways.

First, we deal with the boundary analogue,
that is, with the complex Green operator and the Szeg\"o projection.
Second, we do not require pseudoconvexity and instead focus on
obtaining results for a fixed $q$, $1\leq q \leq n$.  Third, we reduce the regularity hypotheses
in the relationship between the Szeg\"o (resp., Bergman) projection
and the complex Green (resp., $\dbar$-Neumann) operator. Finally,
 we
wanted to establish that the regularity arguments are quite general
and require
only an appropriate weighted Sobolev theory and Hodge-*decomposition.
We provide two examples where the first and third
authors have established the necessary ingredients.

%
%
\section{Proof of Theorem \ref{thm:regularity equivalence} and Proposition \ref{prop:G_q dominates}}\label{sec:proofs}

In \cite{HaRa11}, Harrington and Raich discussed how the regularity of
$G_{q,t} \dbarb$ and $G_{q,t}\dbarbst$ follows from the regularity of
$G_{q,t}$. We provide a proof of this fact for completeness.
\begin{prop}\label{prop:regularity of G dbarbst}
Let $M$ be a smooth CR manifold of hypersurface type that satisfies
the hypotheses of Theorem \ref{thm:regularity equivalence}.
For each $s\geq 0$, there exists $T_s$ so that if $t\geq T_s$ then
$G_{q,t}\dbarb:W^s_{0,q-1}(M) \to W^s_{0,q}(M)$ and
$G_{q,t}\dbarbst:W^s_{0,q+1}(M) \to W^s_{0,q}(M)$ continuously.
\end{prop}

\begin{proof}
We show that $G_{q,t}\dbarbst:W^{s}_{0,q+1}(\Omega) \to
W^{s}_{0,q}(\Omega)$ and $G_{q,t}\dbarb:W^{s}_{0,q-1}(M) \to
W^{s}_{0,q}(M)$ continuously.
The cases $s=0$ and $s=1$ are proven in \cite[Theorem 1.2]{HaRa12}
(see also \cite[Theorem 4.3]{HaRa12}. We can adapt Harrington and
Raich's argument from \cite{HaRa12} for larger $s$.
Observe that
\[
\dbarb \Lambda^s G_{q,t} f = [\dbarb, \Lambda^s] G_{q,t} f  +  \Lambda^s \dbarb G_{q,t} f
\]
and
\[
\dbarbst \Lambda^s G_{q,t} f = [\dbarbst, \Lambda^s] G_{q,t} f  +  \Lambda^s \dbarbst G_{q,t} f.
\]
Implicit in \cite{HaRa11} is the fact that if $\ep>0$, then for $t$ large enough we have
\[
\norm\Lambda^s G_{q,t} f \norm_{t}^2 \leq \ep\Big( \norm\dbarb
\Lambda^s G_{q,t}f\norm_t^2 + \norm\dbarbst \Lambda^s
G_{q,t}f\norm_t^2\Big) + C_t \norm\Lambda^{s-1} G_{q,t} f\norm_t^2.
\]
Since $f$ has smooth coefficients, by choosing $t$ larger (if
necessary), we can use a small constant/large constant argument and
estimate
\[
\norm\Lambda^s G_{q,t} f \norm_{t}^2 \leq \ep\Big( \norm\Lambda^s
\dbarb  G_{q,t}f\norm_t^2 + \norm\Lambda^s \dbarbst
G_{q,t}f\norm_t^2\Big) + C_t \norm\Lambda^{s-1} G_{q,t} f\norm_t^2.
\]

Next, suppose that $f = \dbarbst g$ for a $(0,q+1)$-form with smooth
coefficients. Using induction in $s$ to control $\norm\Lambda^{s-1}G_{q,t}\dbar^*_{b,t}g\norm_t$, we have
\begin{equation}\label{eqn:good Lambda^s G dbarbst g bound}
\norm\Lambda^s G_{q,t}\dbarbst g \norm_t^2
\leq \ep  \norm\Lambda^s \dbarb  G_{q,t} \dbarbst g\norm_t^2 + C_t
\norm\Lambda^{s-1} g\norm_t^2.
\end{equation}
We now handle the term $\norm\Lambda^s \dbarb  G_{q,t} \dbarbst
g\norm_t^2$. Observe that $[\dbar_{b,t}^*,\Lambda^s]=O(\Lambda^s)+tO(\Lambda^{s-1})$.  We adopt the convention that the constant implicit in the error terms is independent of $t$, and we use $C_t$ to represent constants depending on $t$.  We estimate
\begin{align*}
&\norm\Lambda^s \dbarb  G_{q,t} \dbarbst g\norm_t^2\\
& = \big\la\Lambda^s G_{q,t} \dbarbst g , \dbarbst \Lambda^s\dbarb G_{q,t} \dbarbst g \big\ra_t + \big\la[\Lambda^s,\dbarb] G_{q,t} \dbarbst g ,  \Lambda^s\dbarb G_{q,t} \dbarbst g \big\ra_t \\
&\leq \big|\big\la\Lambda^s G_{q,t} \dbarbst g , \Lambda^s\dbarbst \dbarb G_{q,t} \dbarbst g \big\ra_t\big| \\
&\qquad\qquad + O\Big(\norm\Lambda^s G_{q,t}\dbarbst g\norm_t\big(
\norm\Lambda^s\dbarb G_{q,t} \dbarbst g\norm_t +
C_t\norm\Lambda^{s-1}\dbar_b G_{q,t} \dbarbst g\norm_t \big)\Big)\\
&\leq \big|\big\la\Lambda^s G_{q,t} \dbarbst g , \Lambda^s \dbarbst g
\big\ra_t\big| + \frac
12\norm\Lambda^s\dbarb G_{q,t} \dbarbst g\norm_t\\
&\qquad\qquad + O\left(\norm\Lambda^s G_{q,t}\dbarbst g\norm_{t}^2
+ C_t\norm\Lambda^{s-1}\dbar_b G_{q,t} \dbarbst g\norm_t ^2\right).
\end{align*}
Thus, using induction in $s$ to estimate $\norm\Lambda^{s-1}\dbar_b G_{q,t} \dbarbst g\norm_t ^2$,
\[
\norm \Lambda^s\dbarb G_{q,t} \dbarbst g\norm_t^2 \leq
2\big|\big\la\Lambda^s G_{q,t} \dbarbst g , \Lambda^s \dbarbst g
\big\ra_t\big| + C\norm\Lambda^s G_{q,t}\dbarbst g\norm_t^2 +
C_t\norm\Lambda^{s-1} g\norm_t^2.
\]
Next,
\begin{multline*}
\big\la\Lambda^s G_{q,t} \dbarbst g , \Lambda^s \dbarbst g \big\ra_t =
\big\la\Lambda^s \dbarb G_{q,t} \dbarbst g , \Lambda^s g \big\ra_t \\
+ O\Big(\norm\Lambda^s G_{q,t} \dbarbst g\norm_t
\norm\Lambda^{s}g\norm_t+ C_t\norm\Lambda^s G_{q,t} \dbarbst g\norm_t \norm\Lambda^{s-1}g\norm_t\Big).
\end{multline*}
Thus, by absorbing terms after a small constant/large constant argument, we have
\[
\norm \Lambda^s\dbarb G_{q,t} \dbarbst g\norm_t^2 \leq  C\norm\Lambda^s g\norm_t^2 + C\norm\Lambda^s G_{q,t}\dbarbst g\norm_t^2 +  C_t\norm\Lambda^{s-1} g\norm_t^2.
\]
Finally, by choosing $\ep$ sufficiently small in \eqref{eqn:good Lambda^s G dbarbst g bound} to absorb the $\norm\Lambda^s G_{q,t}\dbarbst g\norm_t^2$ terms, we have proven
\[
\norm\Lambda^s G_{q,t}\dbarbst g \norm_t^2
\leq \ep \norm\Lambda^s g\norm_t^2 + C_t \norm\Lambda^{s-1} g\norm_t^2.
\]

The argument to prove
\[
\norm\Lambda^s G_{q,t}\dbarb g \norm_t^2
\leq \ep \norm\Lambda^s g\norm_t^2 + C_t \norm\Lambda^{s-1} g\norm_t^2
\]
is similar, the only difference being that $\dbarbst$ creates lower order terms that depend on $t$, but those are handled with the induction hypothesis and the $C_t \norm\Lambda^{s-1} g\norm_t^2$ term.
This proves the proposition.
\end{proof}

\subsection{Proof of Proposition \ref{prop:G_q dominates}}
 Since $C^\infty_{0,q}(M)$ is dense in $W^k_{0,q}(M)$, it suffices to show the result for $u\in C^\infty_{0,q}(M)$.
Our proof goes by induction. Since $M$ satisfies weak $Y(q)$, the $k=0$ case was proved in \cite{HaRa11}. Assume that the result holds for all $\ell'$ so that $0 \leq \ell'\leq \ell \leq k-1$.
Set $\Lambda^s_\nu = \tilde\zeta_\nu \Lambda^s \zeta_\nu$. Then
\[
\|\dbarb G_q u\|_{W^{\ell+1}(M)}^2 + \|\dbarbs G_q u\|_{W^{\ell+1}(M)}^2
= \sum_{\nu}\big( \|\Lambda^{\ell+1}_\nu  \dbarb G_q u\|_{L^2(M)}^2 + \|\Lambda^{\ell+1}_\nu \dbarbs G_q u\|_{L^2(M)}^2\big).
\]
Examining one term from the sum (call it $RHS$), we first observe that $\dbarbs\Lambda_\nu^{\ell+1}\dbarb G_q  u$ and
$\dbarb\Lambda^{\ell+1}_\nu\dbarbs G_q  u$ are both well-defined terms. For,
\begin{align*}
\Lambda^{\ell+1}_\nu  u + [\dbarbs,\Lambda^{\ell+1}_\nu]& \dbarb G_q  u + [\dbarb,\Lambda^{\ell+1}_\nu] \dbarbs G_q  u \\
&=\Lambda^{\ell+1}_\nu (\dbarb\dbarbs+\dbarbs\dbarb) G_q  u + [\dbarbs,\Lambda^{\ell+1}_\nu] \dbarb G_q  u + [\dbarb,\Lambda^{\ell+1}_\nu] \dbarbs G_q  u \\
&= \dbarbs\Lambda_\nu^{\ell+1}\dbarb G_q  u + \dbarb\Lambda^{\ell+1}_\nu\dbarbs G_q  u,
\end{align*}
so we can make sense of the right-hand side in terms of $\ell+2$
derivatives of $G_q u$ and $\ell+1$ derivatives of $u$, both
well-defined quantities.
We can use integration by parts to observe $RHS$ equals
\begin{align*}
&\big(\dbarbs\Lambda_\nu^{\ell+1}\dbarb G_q  u,\Lambda^{\ell+1}_\nu
G_q  u\big)_0 + \big(\dbarb\Lambda^{\ell+1}_\nu\dbarbs G_q
u,\Lambda^{\ell+1}_\nu G_q  u\big)_0 \\
&\qquad+ O\big(\|\Lambda_\nu^{\ell+1}G_q
u\|_{L^2(M)}(\|\Lambda_\nu^{\ell+1}\dbarb G_q  u\|_{L^2(M)}  +
\|\Lambda_\nu^{\ell+1}\dbarbs G_q  u\|_{L^2(M)} )\big) \\
&= \big(\Lambda^{\ell+1}_\nu  \Boxb  G_q  u,\Lambda^{\ell+1}_\nu G_q  u\big)_0
+ O\big(\|\Lambda_\nu^{\ell+1}G_q  u\|_{L^2(M)}(\|\Lambda_\nu^{\ell+1}\dbarb G_q  u\|_{L^2(M)}  + \|\Lambda_\nu^{\ell+1}\dbarbs G_q  u\|_{L^2(M)} )\big) \\
&= \big(\Lambda^{\ell+1}_\nu (u - H_q  u),\Lambda^{\ell+1}_\nu G_q  u\big)_0\\
&\qquad+ O\big(\|\Lambda_\nu^{\ell+1}G_q  u\|_{L^2(M)}(\|\Lambda_\nu^{\ell+1}\dbarb G_q  u\|_{L^2(M)} + \|\Lambda_\nu^{\ell+1}\dbarbs G_q  u\|_{L^2(M)})\big).
\end{align*}
Using a small constant/large constant argument and the fact that $H_q u=0$, we observe that
\[
\|\dbarb G_q u\|_{W^{\ell+1}(M)}^2 + \|\dbarbs G_q u\|_{W^{\ell+1}(M)}^2
 \leq C_{\ell+1}\big(\|G_q  u\|_{W^{\ell+1}(M)}^2 + \|u\|_{W^{\ell+1}(M)}^2 \big).
\]

For $\dbarbs\dbarb G_q u$ and $\dbarb\dbarbs G_q u$, we also use
induction and an integration by parts argument.
Since $\dbarb\dbarbs \dbarb G_q u = \dbarb(\dbarbs\dbarb G_q +
\dbarb\dbarbs G_q + H_q)u = \dbarb u$, and $u\in W^{k+2}_{0,q}(M)$,
it follows that $ [\Lambda^k_\nu,\dbarb]\dbarbs\dbarb G_q u +
\Lambda^k_\nu\dbarb\dbarbs\dbarb G_q u = \dbarb \Lambda^k _\nu\dbarbs\dbarb G_q \in
W^{k+2}_{0,q}(M)$.
For the induction, the $k=0$ case follows from \cite{HaRa11}. Assume
that the result holds
for all $\ell'$ so that $0 \leq \ell'\leq \ell \leq k-1$.
Therefore, since $\ell+1\leq k$,
\begin{align*}
& \|\Lambda^{\ell+1}_\nu \dbarbs \dbarb G_q u\|_{L^2(M)}^2\\
&= \big(\dbarb \Lambda^{\ell+1}_\nu \dbarbs \dbarb G_q u,
\Lambda^{\ell+1}_\nu  \dbarb G_q u\big)_0
+ O\big(\|\Lambda^{\ell+1}_\nu \dbarbs \dbarb G_q
u\|_{L^2(M)}\|\Lambda_\nu^{\ell+1}\dbarb G_q u\|_{L^2(M)}\big)\\
&= \big(\Lambda^{\ell+1}_\nu \dbarb u, \Lambda^{\ell+1}_\nu  \dbarb G_q u\big)_0
+ O\big(\|\Lambda^{\ell+1}_\nu \dbarbs \dbarb G_q
u\|_{L^2(M)}\|\Lambda_\nu^{\ell+1}\dbarb G_q u\|_{L^2(M)}\big) \\
&= \big(\Lambda^{\ell+1}_\nu u, \Lambda^{\ell+1}_\nu  \dbar_b^*\dbarb G_q u\big)_0
\\
&\qquad\qquad+ O\big(\|\Lambda^{\ell+1}_\nu 
u\|_{L^2(M)}\|\Lambda_\nu^{\ell+1}\dbarb G_q u\|_{L^2(M)}+\|\Lambda^{\ell+1}_\nu \dbarbs \dbarb G_q
u\|_{L^2(M)}\|\Lambda_\nu^{\ell+1}\dbarb G_q u\|_{L^2(M)}\big).
\end{align*}
Using a small constant/large constant argument and the earlier part of
the argument, we may conclude that
\[
\|\dbarbs\dbarb G_q u\|_{W^{\ell+1}(M)}^2 + \|\dbarbs G_q u\|_{W^{\ell+1}(M)}^2
 \leq C_{\ell+1}\big(\|G_q  u\|_{W^{\ell+1}(M)}^2 + \|u\|_{W^{\ell+1}(M)}^2 \big).
\]
A similar argument shows the bound for $\dbarb\dbarbs G_q u$.
\qed

\subsection{Proof of Theorem \ref{thm:regularity equivalence}}

The idea of the proof is simple: the results follow immediately
by expressing $G_q$ in terms of $S_{q-1}$, $S_q$, $S_{q+1}'$, and
weighted operators (that we know
are continuous on $W^s$) and, conversely, by expressing $S_{q-1}$,
$S_q$, $S_{q+1}'$ in terms of $G_q$ and weighted operators.

Let $s\geq 0$. From \cite{HaRa11}, we know that there exists $T_s$ so
that if $t\geq T_s$, then all of the weighted operators:
$G_{q,t}, \dbarb G_{q,t}, \dbarbs G_{q,t}, G_{q,t}\dbarb,
G_{q,t}\dbarbs, I - \dbarbst\dbarb G_{q,t}, S_{q-1,t}, S_{q+1,t}'$ are
continuous in the $W^s$-norm on their respective spaces. The continuity of $I - \dbarbst\dbarb G_{q,t}$ trivially gives
continuity of $\dbarbst\dbarb G_{q,t}$. Also, the argument in \cite[Section 6.6]{HaRa11} implies the
continuity of $\dbarb\dbarbst G_{q,t}$. Moreover, since
\[
H_{q,t} = I - \dbarbst\dbarb G_{q,t} - \dbarb\dbarbst G_{q,t},
\]
it follows that $H_{q,t}$ is continuous in $W^s_{0,q}(M)$. Finally, to
show that $S_{q+1,t}'$ is continuous in $W^s_{0,q+1}$, we note  that
$W^{s+1}_{0,q+1}(M)$ is dense in $W^s_{0,q+1}(M)$ and let $\vp\in
W^{s+1}_{0,q}(M)$. We then observe that
\begin{align*}
\norm\Lambda^s \dbarb G_{q,t}\dbarbst \vp\norm_t^2 =& \Big\la \Lambda^s \dbarbst \dbarb G_{q,t}\dbarbst \vp, \Lambda^s  G_{q,t}\dbarbst \vp \Big\ra_t \\
&+ \Big\la [\dbarbst,\Lambda^s] \dbarb G_{q,t}\dbarbst \vp, \Lambda^s  G_{q,t}\dbarbst \vp \Big\ra_t + \Big\la \Lambda^s  \dbarb G_{q,t}\dbarbst \vp, [\Lambda^s,\dbarb]  G_{q,t}\dbarbst \vp \Big\ra_t.
\end{align*}
Since $\dbarbst\vp$ is $\dbarbst$-closed, it follows that $ \dbarbst \dbarb G_{q,t}\dbarbst \vp = \dbarbst\vp$ so that
\begin{multline*}
\Big\la \Lambda^s \dbarbst \dbarb G_{q,t}\dbarbst \vp, \Lambda^s  G_{q,t}\dbarbst \vp \Big\ra_t
= \Big\la \Lambda^s \dbarbst \vp, \Lambda^s  G_{q,t}\dbarbst \vp \Big\ra_t \\
=  \Big\la \Lambda^s  \vp, \Lambda^s \dbarb G_{q,t}\dbarbst \vp \Big\ra_t +  \Big\la \Lambda^s \vp, [\dbarb,\Lambda^s]  G_{q,t}\dbarbst \vp \Big\ra_t
+  \Big\la [\Lambda^s, \dbarbst] \vp, \Lambda^s  G_{q,t}\dbarbst \vp \Big\ra_t .
\end{multline*}
It now follows that
\[
\norm\Lambda^s \dbarb G_{q,t}\dbarbst \vp\norm_t^2 \leq C_{s,t} \big( \norm\Lambda^s \vp\norm_t \norm\Lambda^s \dbarb G_{q,t}\dbarbst \vp\norm_t + \norm\Lambda^s\vp\norm_t^2\big).
\]
Using a small constant/large constant argument and absorbing terms, we have the continuity of $S_{q+1,t}'$ in $W^s_{0,q+1}(M)$.

We now express $S_{q-1}$, $S_q$, and $S_{q+1}'$ in terms of $G_q$.
For $S_q$, continuity in $W^s$ follows from the formula
\[
S_q = I - \dbarbs\dbarb G_q
\]
and Proposition \ref{prop:G_q dominates}.

Assume that $G_q$ is exactly regular. Assume that $g$ is a $\dbarb$-closed $(0,q-1)$-form.
Then following \cite[Section 5.3]{Str10} (with the zero-order pseudodifferential operator $F_t$ defined in \eqref{eq:F_t} replacing the weight), we have
\begin{align*}
(S_{q-1}f,g)_0 = (f,g)_0 = \la F_t^{-1} f,g \ra_t &= \la S_{q-1,t} F_t^{-1} f,g\ra_t \\
&= (F_t S_{q-1,t} F_t^{-1} f,g)_0 = (S_{q-1}F_t S_{q-1,t} F_t^{-1} f,g)_0.
\end{align*}
Using the fact that $S_{q-1} = I- \dbarbs G_q\dbarb$, it follows that
\begin{align}
S_{q-1} &= S_{q-1}F_t S_{q-1,t} F_t^{-1} = ( I- \dbarbs G_q\dbarb)F_t S_{q-1,t} F_t^{-1}\nn \\
&= F_t S_{q-1,t} F_t^{-1} - \dbarbs G_q[\dbarb, F_t]S_{q-1,t} F_t^{-1} \label{eqn:S_q-1 in terms of S_q-1,t}
\end{align}

For $S_{q+1}'$, we first observe that by \cite[(18)]{HaRa11},
\[
\dbarbst(I-S_{q+1,t}') = \dbarbst - \dbarbst\dbarb G_{q,t}\dbarbst
= \dbarb\dbarbst G_{q,t}\dbarbst + H_{q,t}\dbarbst =  G_{q,t}\dbarb\dbarbst\dbarbst =0.
\]
Next, observe that $S_{q+1,t}' = S_{q+1}'S_{q+1,t}'$, so we write
\begin{align}
S_{q+1}' &= S_{q+1,t}' + S_{q+1}' - S_{q+1}'S_{q+1,t}'
= S_{q+1,t}' + \dbarb G_q \dbarbs (I-S_{q+1,t}') \nn \\
&= S_{q+1,t}' + \dbarb G_q\big( \dbarbs - \dbarbst\big) (I-S_{q+1,t}'). \label{eqn:S_q+1 in terms of weighted}
\end{align}

We now express $G_q$ in terms of $S_{q-1}, S_q$, and $S_{q+1}'$.
We write
\[
G_q = G_q(\dbarb\dbarbs + \dbarbs\dbarb)G_q
= (\dbarbs G_q)^*(\dbarbs G_q) + (G_q\dbarbs)(G_q\dbarbs)^*.
\]
Also,
from \cite[(22)]{HaRa11}, we know that if $\dbarbs\phi=0$, then $\dbarbs G_q\phi=0$, so $\dbarbs(I-S_q) = \dbarbs \dbarbs\dbarb G_q =0$ means that
\begin{align*}
\dbarbs G_q &= \dbarbs G_q S_q
= \dbarbs G_q \big(\dbarb\dbarbst G_{q,t} + \dbarbst\dbarb G_{q,t} + H_{q,t}\big) S_q \\
&= (I-S_{q-1})\dbarbst G_{q,t} S_q + \dbarbs G_q \dbarbst \underbrace{\dbarb G_{q,t} S_q}_{=0} + \dbarbs G_q H_{q,t}  S_q \\
&= (I-S_{q-1})\dbarbst G_{q,t} S_q + \dbarbs G_q H_{q,t}  S_q.
\end{align*}
Note that $S_qG_q\dbarbs=0$ since $(S_q G_q\dbarbs)^* = \dbarb G_qS_q=0$. Also, $\dbarb = S_q\dbarb$ and $H_{q,t} = S_q H_{q,t}$ since $\Ran(\dbarb)\subset\ker(\dbarb)$
and $\dbarb H_{q,t}=0$, respectively. Consequently,
\begin{align*}
G_q \dbarbs  &=  (I-S_q) G_q\dbarbs \\
&= (I-S_q) \big[ G_{q,t}\dbarbst\dbarb + G_{q,t}\dbarb\dbarbst + H_{q,t} \big]G_q \dbarbs \\
&= (I-S_q) G_{q,t} \dbarbst S_{q+1}' + (I-S_q)G_{q,t}\dbarb\dbarbst G_q \dbarbs + (I-S_q)H_{q,t} G_q\dbarbs \\
&= (I-S_q) G_{q,t} \dbarbst S_{q+1}' + \underbrace{(I-S_q)S_q}_{=0}\dbarb\dbarbst G_{q,t} G_q \dbarbs
+ \underbrace{(I-S_q)S_q}_{=0} H_{q,t} G_q\dbarbs \\
&= (I-S_q) G_{q,t} \dbarbst S_{q+1}'.
\end{align*}
We also need to control $(\dbarbs G_q)^*$ and $(G_q\dbarbs)^*$. If $T_t$ is a continuous operator on $L^2(M,e^{-t|z|^2})$ then we can compute its adjoint in $L^2(M)$ as follows:
\[
(T_tf,g)_0 = \la T_tf, F_t^{-1}g\ra_t = \la f, T_t^* F_t^{-1} g\ra_t = (f, F_t T_t^* F_t^{-1}g)_0,
\]
and we observe that the adjoint of $T_t$ is $F_t T_t^* F_t^{-1}$. We therefore compute
\[
(\dbarbs G_q)^* = F_t S_q G_{q,t}\dbarb (I-S_{q-1})F_t^{-1} + F_t S_q H_{q,t}G_q\dbarb F_t^{-1}
\]
and
\[
(G_q\dbarbs)^* = F_t S_{q+1}' \dbarb (I-S_q) F_t^{-1}.
\]

We now investigate the harmonic projection $H_{q,t}$. From \cite[p.156]{HaRa11}, we know that for a $(0,q)$-form $\vp$,
\[
\|\vp\|_{W^s(M)}^2 \leq C_t\big( \|\dbarb\vp\|_{W^s(M)}^2 + \|\dbarbst\vp\|_{W^s(M)}^2 + \|\vp\|_{W^{s-1}(M)}^2\big).
\]
By density, this means that for any $f\in L^2_{0,q}(M)$,
\[
\|H_{q,t} f\|_{W^s(M)}^2 \leq C_{t,s} \|H_{q,t} f\|_{L^2(M)}^2 \leq C_{t,s} \|f\|_{L^2(M)}^2.
\]
Therefore, if $t\geq T_{s+1}$, then $H_{q,t}:L^2_{0,q}(M)\to W^{s+1}_{0,q}(M)$ and
\begin{multline*}
\| \dbarbs G_q H_{q,t}  S_q f\|_{W^s(M)} \leq C \|G_q H_{q,t}  S_q f\|_{W^{s+1}(M)} \\
\leq C \|H_{q,t}  S_q f\|_{W^{s+1}(M)} \leq C_{s,t} \|H_{q,t}  S_q f\|_{L^2(M)} \leq C_{s,t} \|f\|_{L^2(M)},
\end{multline*}
so
$\dbarbs G_q H_{q,t}  S_q: L^2_{0,q}(M) \to W^s_{0,q-1}(M)$.\qed

%
%
\section{Stein Manifolds}\label{sec:Stein}
Finally, we briefly indicate how to adapt the argument to prove our main
result in the case of a Stein manifold.  We need the following result.

\begin{thm}\label{thm:weighted theory}
Let $M$ be an $n$-dimensional Stein manifold and $\Omega\subset M$ be
a bounded subset with a smooth boundary satisfying weak $Z(q)$ for
some $1\leq q \leq n-1$.
Then there exists $\tilde t>0$ such that for all $t>\tilde t$  and $s \geq - \frac 12$ we have
  \begin{enumerate}
    \item The weighted $\dbar$-Neumann operator $N_{q,t}$ exists and
      is continuous in $W^s_{0,q}(\Omega)$.

    \item The canonical solution operators to $\dbar$ given by\\
    $\dbar^*_t N_t^q:W^s_{0,q}(\Omega)\rightarrow W^s_{0,q-1}(\Omega)$
    and $N_t^q\dbar^*_t:W^s_{0,q+1}(\Omega)\rightarrow
    W^s_{0,q}(\Omega)$ are continuous.

    \item The canonical solution operators to $\dbar^*_t$ given by\\
    $\dbar N_t^q:W^s_{0,q}(\Omega)\rightarrow W^s_{0,q+1}(\Omega)$ and
    $N_t^q\dbar:W^s_{0,q-1}(\Omega)\rightarrow W^s_{0,q}(\Omega)$ are
    continuous.

    \item For every $f\in W^s_{0,q}(\Omega)\cap\ker\dbar$ there exists
      $u\in W^s_{0,q-1}(\Omega)$ such that $\dbar u=f$.
  \end{enumerate}
\end{thm}

In \cite{HaRa12}, Harrington and Raich proved  Theorem
\ref{thm:weighted theory} for $-\frac 12 \leq s \leq 1$. Standard techniques show
 that their arguments extend seamlessly to all $s\geq -\frac 12$.\medskip

The proofs of the results in Section \ref{subsec:Stein mfld results}
are now straightforward, given the proofs of Section \ref{subsec:CR
  mfld results} and
\cite[Section 5.3]{Str10}. The general outline of the argument is
contained in \cite[Section 5.3]{Str10}. Our hypotheses allow us to
prove closed range and Kohn's weighted theory for a fixed $q$,
$1 \leq q \leq n-1$. Using the arguments from the proofs of the
results in Section \ref{subsec:CR mfld results} with the weighted
theory from Theorem \ref{thm:weighted theory}, $L^2$ theory from
\cite{HaRa12}, and the recognition that the tangential derivatives
control the Sobolev norms (so we can replace the $\Lambda^k$ terms
with $D_{T^\alpha}$), we can repeat the arguments above to
prove the results in Section \ref{subsec:Stein mfld results}.

\bibliographystyle{alpha}
\bibliography{mybib3-26-13}

\begin{thebibliography}{ABZ06}

\bibitem[ABZ06]{AhBaZa06}
H.~Ahn, L.~Baracco, and G.~Zampieri.
\newblock Non-subelliptic estimates for the tangential {C}auchy-{R}iemann
  system.
\newblock {\em Manuscripta Math.}, 121(4):461--479, 2006.

\bibitem[Bog91]{Bog91}
A.~Boggess.
\newblock {\em CR Manifolds and the Tangential Cauchy-Riemann Complex}.
\newblock Studies in Advanced Mathematics. CRC Press, Boca Raton, Florida,
  1991.

\bibitem[BS86]{BoSh86}
H.\ Boas and M.-C.\ Shaw.
\newblock Sobolev estimates for the {L}ewy operator on weakly pseudoconvex
  boundaries.
\newblock {\em Math.\ Ann.}, 274:221--231, 1986.

\bibitem[BS90]{BoSt90}
H.\ Boas and E.\ Straube.
\newblock Equivalence of regularity for the {B}ergman projection and the
  {$\overline \partial$}-{N}eumann operator.
\newblock {\em Manuscripta Math.}, 67(1):25--33, 1990.

\bibitem[CS01]{ChSh01}
S.-C.\ Chen and M.-C.\ Shaw.
\newblock {\em Partial Differential Equations in Several Complex Variables},
  volume~19 of {\em Studies in Advanced Mathematics}.
\newblock American Mathematical Society, 2001.

\bibitem[FK72]{FoKo72}
G.~B.\ Folland and J.~J.\ Kohn.
\newblock {\em The {Neumann} problem for the {Cauchy}-{Riemann} Complex},
  volume~75 of {\em Ann.\ of Math.\ Stud.}
\newblock Princeton University Press, Princeton, New Jersey, 1972.

\bibitem[H{\"o}r65]{Hor65}
L.~H{\"o}rmander.
\newblock ${L}^{2}$ estimates and existence theorems for the $\bar \partial $
  operator.
\newblock {\em Acta Math.}, 113:89--152, 1965.

\bibitem[HR]{HaRa12}
P.~Harrington and A.~Raich.
\newblock Closed range for $\bar\partial$ and $\bar\partial_b$ on bounded
  hypersurfaces in {S}tein manifolds.
\newblock {\em submitted}.
\newblock arXiv:1106.0629.

\bibitem[HR11]{HaRa11}
P.~Harrington and A.~Raich.
\newblock Regularity results for $\bar\partial_b$ on {CR}-manifolds of
  hypersurface type.
\newblock {\em Comm.\ Partial Differential Equations}, 36(1):134--161, 2011.

\bibitem[Koh73]{Koh73}
J.~J.\ Kohn.
\newblock Global regularity for {$\bar \partial $} on weakly pseudo-convex
  manifolds.
\newblock {\em Trans.\ Amer.\ Math.\ Soc.}, 181:273--292, 1973.

\bibitem[Koh86]{Koh86}
J.J.\ Kohn.
\newblock The range of the tangential {C}auchy-{R}iemann operator.
\newblock {\em Duke Math.\ J.}, 53:525--545, 1986.

\bibitem[Nic06]{Nic06}
A.~Nicoara.
\newblock Global regularity for $\bar\partial_b$ on weakly pseudoconvex {CR}
  manifolds.
\newblock {\em Adv. Math.}, 199:356--447, 2006.

\bibitem[Rai10]{Rai10}
A.~Raich.
\newblock Compactness of the complex {G}reen operator on {CR}-manifolds of
  hypersurface type.
\newblock {\em Math.~Ann.}, 348(1):81--117, 2010.

\bibitem[RS08]{RaSt08}
A.\ Raich and E.\ Straube.
\newblock Compactness of the complex {G}reen operator.
\newblock {\em Math.\ Res.\ Lett.}, 15(4):761--778, 2008.

\bibitem[Sha85]{Sha85}
M.-C.\ Shaw.
\newblock ${L}\sp 2$-estimates and existence theorems for the tangential
  {C}auchy-{R}iemann complex.
\newblock {\em Invent.\ Math.}, 82:133--150, 1985.

\bibitem[Str10]{Str10}
E.\ Straube.
\newblock {\em Lectures on the ${\mathcal{L}}^2$-Sobolev Theory of the
  $\bar\partial$-Neumann Problem}.
\newblock ESI Lectures in Mathematics and Physics. European Mathematical
  Society (EMS), Z{\"u}rich, 2010.

\bibitem[Zam08]{Zam08}
G.~Zampieri.
\newblock {\em Complex analysis and {CR} geometry}, volume~43 of {\em
  University Lecture Series}.
\newblock American Mathematical Society, Providence, RI, 2008.

\end{thebibliography}

\end{document}